\documentclass[11pt]{article}
\usepackage{e-jc}
\usepackage{amsthm,amsmath,amssymb}

\usepackage[colorlinks=true,citecolor=black,linkcolor=black,urlcolor=blue]{hyperref}

\usepackage{graphicx}

\numberwithin{equation}{section}
\numberwithin{figure}{section}
\theoremstyle{plain}
\newtheorem{thm}{Theorem}[section]
\newtheorem{lem}[thm]{Lemma}

\theoremstyle{remark}

\newcommand{\M}{\operatorname{M}}

\newcommand{\wt}{\operatorname{\textbf{wt}}}
\newcommand{\tr}{\operatorname{tr}}

\title{A New Proof for a Triple Product Formula for Plane Partitions}
\author{Tri Lai\footnote{This research was supported in part by the Institute for Mathematics and its Applications with funds provided by the National Science Foundation (grant no. DMS-0931945).}\\
\small Department of Mathematics\\[-0.8ex]
\small University of Nebraska -- Lincoln\\[-0.8ex]
\small Lincoln, NE 68588\\
\small Email: \tt tlai3@unl.edu\\
}

\date{\small Mathematics Subject Classifications: 05A15,  05C30, 05C70}

\begin{document}
\maketitle

\begin{abstract}
 Stanley generalized MacMahon's classical theorem by proving a product formula for the norm-trace generating function for plane partition with unbounded parts. In his recent work on biothorgonal polynomials, Kamioka proved a finite analogue of Stanley's formula for plane partitions with bounded parts  (\href {http://arxiv.org/abs/1508.01674}
  {\path{arXiv:1508.01674}}).  In this paper, we use techniques from the enumeration of tilings to give a new proof for Kamioka's formula.

\end{abstract}

\section{Introduction}
A \emph{partition} $\lambda$ of $n$ to be non-increasing sequence of  positive integers $(\lambda_1,\lambda_2,\dotsc, \lambda_r)$ such that
 $\lambda_1+ \lambda_2+\dotsc+ \lambda_r=n$. Given a partition $\lambda=(\lambda_1,\lambda_2,\dotsc, \lambda_r)$ of $n$, a \emph{plane partition} of $n$ with the shape $\lambda=(\lambda_1,\lambda_2,\dotsc, \lambda_r)$ is an array of non-negative integers of the form
\begin{equation}
\pi=\begin{tabular}{rccccccccc}
$\pi_{1,1}$   &$\pi_{1,2}$                 &$\pi_{1,3}$               & $\dotsc$               &  $\dotsc$                        & $\dotsc$                            &   $\pi_{1,\lambda_1}$ \\\noalign{\smallskip\smallskip}
$\pi_{2,1}$   &  $\pi_{2,2}$              & $\pi_{2,3}$             &  $\dotsc$               & $\dotsc$                        &         $\pi_{2,\lambda_2}$&          \\\noalign{\smallskip\smallskip}
$\vdots$    &       $\vdots$            & $\vdots$                &        $\vdots$         &     \reflectbox{$\ddots$\quad}               &    &              \\\noalign{\smallskip\smallskip}
 $\pi_{r,1}$  &  $\pi_{r,2}$               & $\pi_{r,3}$              &     $\dotsc$             &   $\pi_{r,\lambda_r}$ &                                          &           \\\noalign{\smallskip\smallskip}
\end{tabular},
\end{equation}
where all rows and columns are weakly decreasing from left to right and from top to bottom, respectively. The \emph{norm} (or  the \emph{volume}) $|\pi|$ of the plane partition $\pi$ is defined to be the sum of all entries in $\pi$. The entries in $\pi_{i,j}$ are called the \emph{parts} of $\pi$.

Denote by  $\mathcal{P}(r,c,n)$ the set of all plane partitions having at most $r$ rows, $c$ columns, and the  maximal part $\pi_{1,1}$ at most $n$. MacMahon \cite{Mac} showed that the \emph{norm generating function} of plane partitions in $\mathcal{P}(r,c,n)$ is given by the following triple product
\begin{equation}\label{Macformula1}
\sum_{\pi\in \mathcal{P}(r,c,n)}q^{|\pi|}=\prod_{i=1}^{r}\prod_{i=1}^{c}\prod_{i=1}^{n}\frac{1-q^{i+j+k-1}}{1-q^{i+j+k-2}}.
\end{equation}
By letting $n\rightarrow \infty$, we obtain the norm generating function for the plane partitions with unbounded parts
\begin{equation}\label{Macformula2}
\sum_{\pi\in \mathcal{P}(r,c)}q^{|\pi|}=\prod_{i=1}^{r}\prod_{i=1}^{c}\frac{1}{1-q^{i+j-1}},
\end{equation}
where $\mathcal{P}(r,c)$ denotes the set of plane partitions with at most $r$ rows and $c$ columns.

Stanley \cite{Stanley2} generalized (\ref{Macformula2}) by introducing the \emph{trace} of the plane partition $\pi$
\begin{equation}
\tr(\pi):=\sum_{i=1} \pi_{i,i}
\end{equation}
and proving a closed form product formula for the \emph{norm-trace generating function}
\begin{equation}\label{staneq}
\sum_{\pi\in\mathcal{P}(r,c)}q^{|\pi|}a^{\tr(\pi)}=\prod_{i=1}^{r}\prod_{i=1}^{c}\frac{1}{1-aq^{i+j-1}}.
\end{equation}

A \emph{Young diagram} $[\lambda]$ of a partition $\lambda=(\lambda_1,\lambda_2,\dotsc, \lambda_r)$ (or of shape $\lambda=(\lambda_1,\lambda_2,\dotsc, \lambda_r)$) is a collection of unit squares
$(i, j)$ on a square grid $\mathbb{Z}^2$, with $1 \leq i \leq r$ and $1 \leq  j \leq  \lambda_i$ (see Figure \ref{YoungD}(a)).
 We usually write the entries of a plane partition in a Young diagram of the same shape (see Figure \ref{YoungD}(b)).

\begin{figure}\centering
\includegraphics[width=12cm]{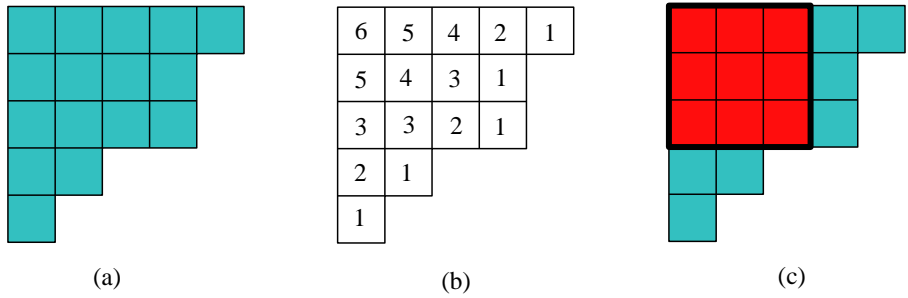}
\caption{(a) Young diagram of shape (5,4,4,2,1). (b) A plane partition of the same shape. (c) Durfee square of a Young diagram.}
\label{YoungD}
\end{figure}
The largest square $D = \{(i, j), 1 \leq i, j \leq r\}$ fitting in the Young diagram $[\lambda]$ is called the \emph{Durfee square} of the Young diagram $[\lambda]$ (or of the partition $\lambda$) (see the square restricted by the bold contour in Figure \ref{YoungD}(c))

For a plane partition $\pi$ of shape $\lambda$ and a number $1\leq k\leq \pi_{1,1}$, we define the  \emph{$k$-truncation} $\pi^{(k)}$ of $\pi$ to be the plane partition obtained from $\pi$ by removing all parts less than $k$ (see the plane partitions restricted in the contour in the left pictures in Figures \ref{YoungD2}(a)--(e)). We call the partition corresponding to the shape $\lambda$ of the plane partition $\pi^{(k)}$ the \emph{$k$-cross section} of $\pi$, denoted by $\lambda^{(k)}(\pi)$ (see the right pictures in Figures \ref{YoungD2}(a)--(e) for $k=2,3,4,5,6$, respectively). It is obvious that $\pi^{(1)}=\pi$ and $\lambda^{(1)}(\pi)=\lambda$.

\begin{figure}\centering
\includegraphics[width=12cm]{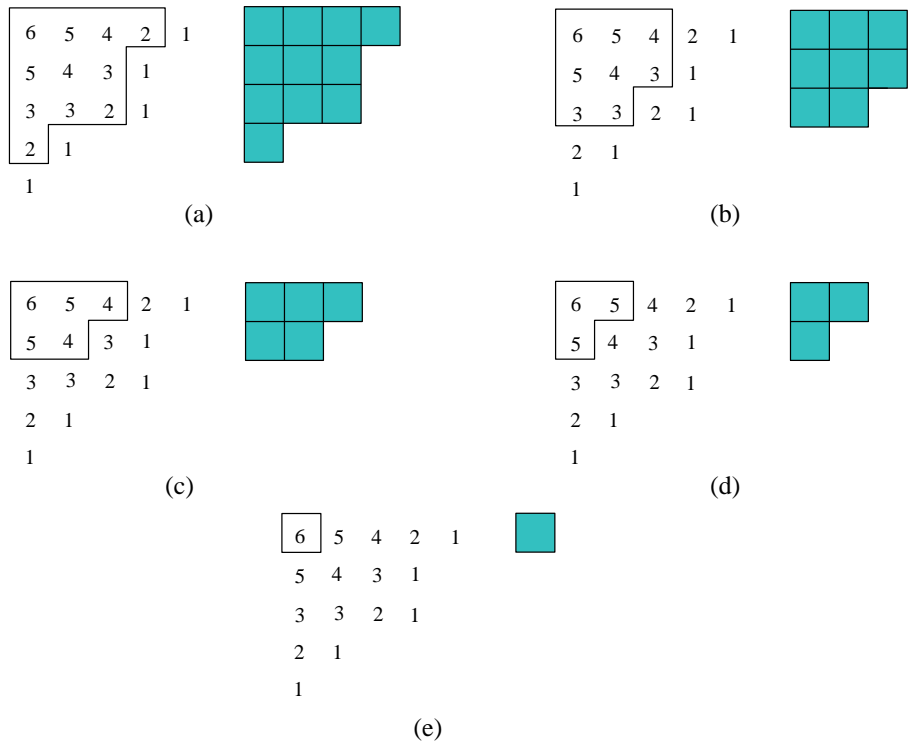}
\caption{The $k$-truncations and the corresponding Young diagram of the $k$-cross sections of the  plane partition in Figure \ref{YoungD}, for (a) $k=2$, (b) $k=3$, (c) $k=4$, (d) $k=5$, and (e) $k=6$.}
\label{YoungD2}
\end{figure}

Recently, Kamioka \cite{Kam} uses biorthogonal polynomials and lattice paths to prove a striking finite analogue of Stanley's formula (\ref{staneq}) for plane partitions in $\mathcal{P}(r,c,n)$.
\begin{thm}\label{kamiokathm} Assume that $r,c,n$ are nonnegative integers. Then
\begin{equation}\label{Kamiokaeq}
\sum_{\pi}q^{|\pi|}a^{\tr(\pi)}\prod_{k=1}^{\pi_{1,1}}\frac{(q^{n-k+1};q)_{D_{k}(\pi)}}{(aq^{n-k+1};q)_{D_{k}(\pi)}}
=\prod_{i=1}^{r}\prod_{j=1}^{c}\prod_{k=1}^{n}\frac{1-aq^{i+j+k-1}}{1-aq^{i+j+k-2}},\end{equation}
where
\[(x;q)_N=(1-x)(1-xq)\dots(1-xq^{N-1})\]
and where $D_{k}(\pi)$ is the size of the Durfee square of the $k$-truncation $\lambda^{(k)}(\pi)$ of $\pi$.
\end{thm}

One readily sees that Kamioka's formula also implies MacMahon's classical formula (\ref{Macformula1}) by specifying $a=1$.

%We notice that $tr(\pi)$ is exactly the sum of all $D_{k}(\pi)$'s, for $k=1,2,\dotsc,\pi_{1,1}$.
\bigskip

The plane partitions in $\mathcal{P}(r,c,n)$ can be viewed as piles (or stacks) of unit cubes fitting in a $r\times c \times n$ box (see Figures \ref{BoxedPP}(a)--(c)). The latter in turn are in bijection with the \emph{lozenge tilings} of a semi-regular  hexagon $H_{r,c,n}$ of sides $r,c,n,r,c,n$ (in cyclic order) on the triangular lattice (see Figures \ref{BoxedPP}(c) and (d)).  Here a \emph{lozenge} is the union of any two unit equilateral triangles sharing an edge, and a \emph{lozenge tiling} of a region is a covering of the region by lozenges so that there are no gaps and overlaps. In the view of this,  MacMahon's identity (\ref{Macformula1}) and Kamioka's identity (\ref{Kamiokaeq}) give triple product formulas for weighted sums of lozenge tilings of the semi-regular hexagon $H_{r,c,n}$.  We refer the reader to \cite{Tri1} and \cite{Tri2} for more similar generalizations of the MacMahon formula (\ref{Macformula1}).

The goal of this paper is using techniques  in enumeration of tilings, in particular the graphical condensation introduced by Eric Kuo \cite{Kuo}, to give a new proof for the triple-product formula (\ref{Kamiokaeq}).

\begin{figure}\centering
\includegraphics[width=12cm]{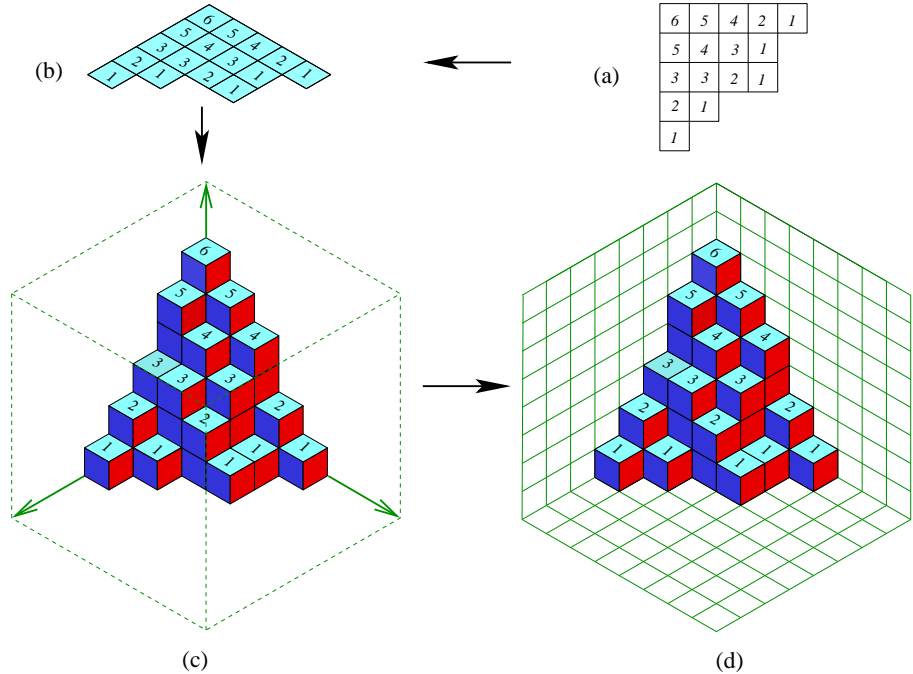}
\caption{Bijection between plane partitions in $\mathcal{P}(8,8,8)$ and lozenge tilings of the hexagon $H_{8,8,8}$.}
\label{BoxedPP}
\end{figure}

\section{Preliminaries}

Lozenges of a region can carry weights; and we define the \emph{weight} of a lozenge tiling of the region to be the weight product of its lozenges. In the weighted case, we use the notation $\M(R)$ for the sum of weights of all lozenge tilings in $R$. We call $\M(R)$ a \emph{tiling generating function} of $R$.

A \emph{forced lozenge} of a region is the lozenge contained in any tiling of the region. Assume that one removes several forced lozenges $\ell_1,\ell_2,\dots,\ell_k$ from a region $R$ whose weights are $\wt(\ell_1),\wt(\ell_2),\dots,\wt(\ell_k)$. We get a new region $R'$ and obtain
\begin{equation}\M(R)=\M(R')\prod_{i=1}^{k}\wt(\ell_i).\end{equation}

A \emph{(perfect) matching} of a graph $G$ is a collection of edges that cover each vertex exactly once. The \emph{dual graph} $G$ of a region $R$ on the triangular lattice is the graph whose vertices are the unit triangles in $R$ and whose edges connect precisely two unit triangles sharing an edge. If lozenges in $R$ are weighted, we assign to each edge in the dual graph $G$ the same weight as that of  its corresponding lozenge. Thus, one can identify the tilings of $R$ with the matchings of its dual graph $G$. In the view of this, we use the same notation $\M(G)$ for the sum of weights of all matchings in $G$, where the weight of a matchings is the weight product of its edges. We usually call $\M(G)$ a \emph{matching generating function} of $G$.

In 2004, Eric Kuo \cite{Kuo} used a combinatorial interpretation of the well-known \emph{Dodgson condensation} (or \emph{Desnanot-Jacobi identity}, see e.g. \cite{Dod}) to (re)prove the Aztec diamond theorem by Elkies, Kuperberg, Larsen and Propp \cite{Elkies1, Elkies2}. In particular, he proved the following condensation:
\begin{thm}[Theorem 5.1 in \cite{Kuo}]\label{Kuothm}
Let $G=(V_1,V_2,E)$ be a (weighted) planar bipartite graph with $|V_1|=|V_2|$. Assume that $u,v,w,s$ are four vertices appearing in a cyclic order on a face of $G$. Assume in addition that $u,w\in V_1$ and $v,s\in V_2$. Then
\begin{equation}\label{Kuoeq}
\M(G)\M(G-\{u,v,w,s\})=\M(G-\{u,v\})\M(G-\{w,s\})+\M(G-\{u,s\})\M(G-\{v,w\}).
\end{equation}
\end{thm}
This theorem is the key of our proof in the next section.

\section{New proof of Theorem \ref{kamiokathm}}

Lozenges have three orientations: \emph{left, right} and \emph{horizontal} as in Figure \ref{lozengehorizontal}. We call the vertical line passing the north vertex of the semi-regular hexagon $H_{r,c,n}$ the \emph{axis} of the hexagon (see the dotted line in Figure \ref{Kuotrace2}).

\begin{figure}\centering
\includegraphics[width=6cm]{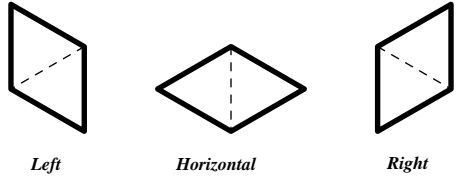}
\caption{Three orientations of lozenges.}
\label{lozengehorizontal}
\end{figure}

We first encode each plane partition $\pi$ in $\mathcal{P}(r,c,n)$  as a lozenge tiling of the semi-regular hexagon $H_{r,c,n}$ whose weight is $a^{\tr(\pi)}q^{|\pi|}$ as follows. View $\pi$ as a pile of $|\pi|$ unit cubes fitting in a $r \times c \times n$ box as in Figure \ref{BoxedPP}. Here, each horizontal lozenge in the tiling corresponding to $\pi$ is pictured as the top of a column of unit cubes. We assign to each horizontal lozenge a weight $q^h$, where $h$ is the height of its corresponding column, except for the ones on the axis that are assigned a weight $(aq)^h$. All remaining lozenges are weighted by $1$. It is easy to see that the weight of the tiling is $a^{\tr(\pi)}q^{|\pi|}$. This weight assignment is called the \emph{natural $q$-weight assignment} of the lozenge tiling.

For each horizontal lozenge $\ell$, we draw the vertical line passing the top and the bottom vertices of $\ell$. Assume that this vertical line passing $t$ other vertical lozenges above $\ell$ before meeting  the northeast or northwest side of the hexagon. We say that $t$ is the \emph{depth} of the horizontal lozenge $\ell$ (see Figure \ref{height}). It is easy to see that the depth $t$ of a horizontal lozenge on the $i$ lattice path and the height $h$ of the corresponding column of unit cubes are related by the following identity
\[t=n-h+(i-1).\]

\begin{figure}\centering
\includegraphics[width=5cm]{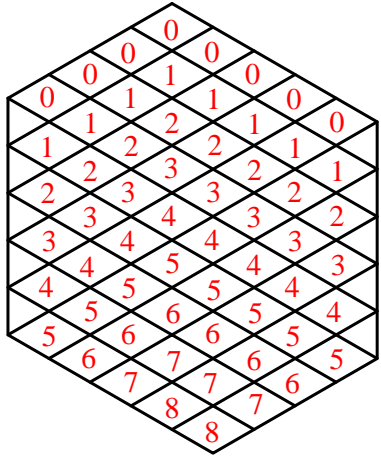}
\caption{The depth of the lozenges in a hexagon.}
\label{height}
\end{figure}

One readily sees that the above weight assignment \emph{depends} on the choice of the tilings. We would like to modify this weight assignment as follows.

First, we encode each tiling as a $r$-tuple of  non intersecting lozenge paths connecting the northwest and southeast sides of the hexagon (see Figure \ref{lozengepath}; the lozenge paths consisting horizonal and right lozenges running along the dotted paths). Multiply the weight of horizontal lozenges on the $i$-th path (ordered from top to bottom) by $q^{-n-i+1}i$, the one intersect the vertical axis by $(aq)^{-n-i+1}$. Second,  we multiply the weight of a horizontal lozenge on the $i$-th path that is passed through by the vertical axis by $\frac{(q^{n};q^{-1})_{n-t_i}}{(aq^{n};q^{-1})_{n-t_i}}$, where $t_i$ is the depth the lozenge. This way each horizontal lozenge is weighted by $q^{-t}$, and the one on the axis is weighted by $(aq)^{-t}\frac{(q^{n};q^{-1})_{n-t}}{(aq^{n};q^{-1})_{n-t}}$, where $t$ is the distance between the top of the horizontal lozenge and the southwest side of the hexagon. We denote by $\wt$ this new weight assignment; and this weight assignment does \emph{not} depend on the choice of tiling.

\begin{figure}\centering
\includegraphics[width=6cm]{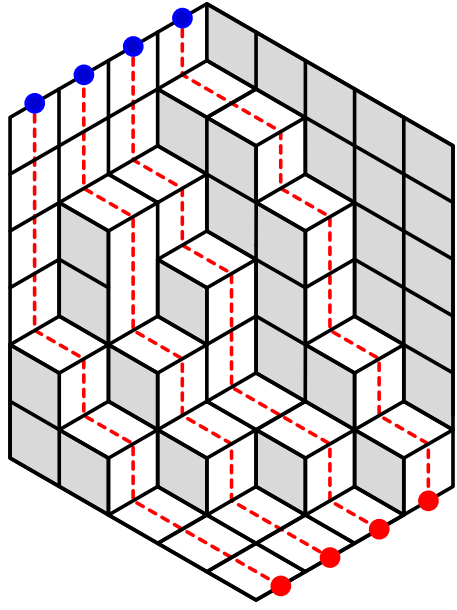}
\caption{Encode each lozenge tiling of the hexagon $H_{3,5,6}$ as a $3$-tuples of distinct lozenge paths.}
\label{lozengepath}
\end{figure}

\begin{lem}
Let $\pi$ be any plane partition in $\mathcal{P}(r,c,n)$, and $T_{\pi}$ be the lozenge tiling of the hexagon $H_{r,c,n}$ corresponding to $\pi$. Then
\begin{equation}
A\wt(T_{\pi})=a^{-d(d+1)/2-(n+1)d}q^{-c(r(r+1)/2+(n-1)r)} q^{|\pi|}a^{tr(\pi)}w_n(\pi;a;q),
\end{equation}
where $w_n(\pi,;a;q)=\prod_{k=1}^{\pi_{1,1}}\frac{(q^{n-k+1};q)_{D_{k}(\pi)}}{(aq^{n-k+1};q)_{D_{k}(\pi)}}$ as defined in Kamioka's Theorem \ref{kamiokathm}, $d=\min (r,c)$ is the number of horizontal lozenges on the vertical axis, and
\[A=\prod_{i=1}^{d}\frac{(aq^{n+i-1};q^{-1})_{i-1}}{(q^{n+i-1};q^{-1})_{i-1}}.\]
\end{lem}
\begin{proof}
 We first assign the natural weight assignment on the lozenges of the tiling $T_{pi}$. The weight of the tiling $T_{\pi}$ is now $q^{|\pi|}a^{tr(\pi)}$. Next, we investigate how the weight of $T_{\pi}$ changes when we convert the natural weight of $T_{\pi}$ into the weight assignment $\wt$.

Consider the $r$-tuple of lozenge paths corresponding to the tiling $T_{\pi}$. Each lozenge path has exactly $c$ right lozenges, in which there is exactly one passed through by the vertical axis.  Thus, the factor for the first step in the above weight modification is $a^{-d(d+1)/2-(n+1)d}q^{-c(r(r+1)/2+(n-1)r)}$.

Let us consider the factor for the second step.
We write each factor $\frac{(q^{n-k+1};q)_{D_{k}(\pi)}}{(aq^{n-k+1};q)_{D_{k}(\pi)}}$ in $w_n(\pi,;a;q)$ as
\begin{align}
w_n(\pi,a,q)&=\prod_{k=1}^{\pi_{1,1}}\frac{(q^{n-k+1};q)_{D_{k}(\pi)}}{(aq^{n-k+1};q)_{D_{k}(\pi)}}=\prod_{k=1}^{\pi_{1,1}}\prod_{i=1}^{D_{k}(\pi)}\frac{1-q^{n-k+i}}{1-aq^{n-k+i}}\notag\\
&=\prod_{i=1}^{d}\prod_{k=1}^{h_i}\frac{1-q^{n-k+i}}{1-aq^{n-k+i}}=\prod_{i=1}^{d}\frac{(q^{n+i-1};q^{-1})_{h_i}}{(aq^{n+i-1};q^{-1})_{h_i}},\end{align}
where $h_i$ the height of the column corresponding to the $i$-th horizontal lozenge on the vertical axis (ordered from the top to the bottom).

 There are exactly $d=\min(r,c)$ right lozenges on the axis . The $i$-th one ($i=1,2,\dotsc,d$) from the top has weight
 \begin{align}
 \frac{(aq^{n};q^{-1})_{n-t_i}}{(aq^{n};q^{-1})_{n-t_i}}= \frac{(q^{n+i-1};q^{-1})_{h_i}}{(aq^{n+i-1};q^{-1})_{h_i}} \cdot \frac{(aq^{n+i-1};q^{-1})_{i-1}}{(q^{n+i-1};q^{-1})_{i-1}}.
 \end{align}
This implies that product of all the $d$ horizontal lozenges along the vertical axis is exactly $w_n(\pi,a,q)\times A$, where $A=\prod_{i=1}^{d}\frac{(aq^{n+i-1};q^{-1})_{i-1}}{(q^{n+i-1};q^{-1})_{i-1}}$.

Combining the factors in the two steps, we get
\[wt(T_{\pi})=A a^{-d(d+1)/2-(n+1)d}q^{-c(r(r+1)/2+(n-1)r)}q^{|\pi|}a^{tr(\pi)}w_n(\pi;a;q)\]
\end{proof}

By the above lemma, in order to prove Kamioka's theorem, we need to show that the tiling generating function of the hexagon $H_{r,c,n}$ weighted by $\wt$ is given by
\begin{equation}\label{main}
\M(H_{r,c,n})=a^{-d(d+1)/2-(n+1)d}q^{-c(r(r+1)/2+(n-1)r)}A \prod_{i=1}^{r}\prod_{j=1}^{c}\prod_{k=1}^{n}\frac{1-aq^{i+j+k-1}}{1-aq^{i+j+k-2}}.
\end{equation}

\medskip

We prove (\ref{main}) by induction on $r+c+n$. The base cases are the situations when at least one of $r,c,n$ equals to $0$.
When $c=0$ or $r=0$, then $d=0$. It implies that all terms in the right-hand side of (\ref{main}) are 1, and it is easy to see that the left-hand side is also 1.
When $n=0$, the hexagon $H_{r,c,0}$ has only one tiling consisting all horizontal lozenges. Then the right-hand side of  (\ref{main}) is simply the weight product of all the horizontal lozenges, which is
\begin{align}
a^{-d(d+1)/2-(n-1)d}q^{-c(r(r+1)/2+(n-1)r)}&\prod_{i=1}^{d}\frac{(aq^{n+i-1};q^{-1})_{i-1}}{(q^{n+i-1};q^{-1})_{i-1}}\\
&=a^{-d(d+1)/2-(n+1)d}q^{-c(r(r+1)/2+(n-1)r)}A.
\end{align}
This verifies the formula (\ref{main}) in this case.

\bigskip

For induction step, we assume that (\ref{main}) holds for any hexagon $H(r,c,n)$ in which the $r$-, $c$- and $n$-parameters are positive and whose sum is strictly less than $r+c+n$.

We apply Kuo condensation Theorem \ref{Kuothm} to the dual graph $G$ of the hexagon $H_{r,c,n}$ weighted by $\wt$ as in Figure \ref{Kuotrace}; the black triangles indicate the unit triangles corresponding to the four vertices $u,v,w,s$. In particular, the $u$-triangle is the black one on the southwest corner of the hexagon, and the $v$-, $w$-, and $s$-triangles are the black ones appearing respectively when we go counter-clockwise from the $u$-triangle.

\begin{figure}\centering
\includegraphics[width=6cm]{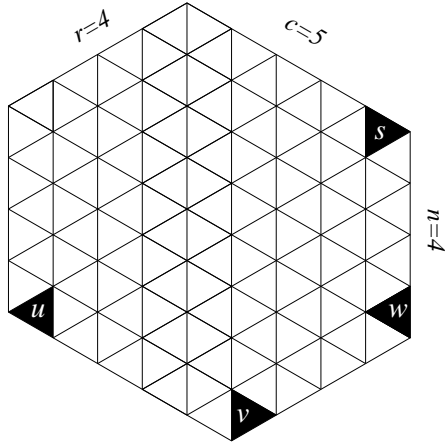}
\caption{How we apply Kuo condensation.}
\label{Kuotrace}
\end{figure}

Let us  consider the region corresponding to the graph $G-\{u,v,w,s\}$. It has several forced lozenges. All of the forced lozenges have weight 1, except for $r-1$ horizontal ones long the southeast side of the hexagon, whose weights are $q^{-n},q^{-(n+1)},\dotsc,q^{-(n+r-2)}$ as they appear from right to left. Removing these forced lozenges, we get the hexagon $H_{r-1,c-1,n}$ weighted by $\wt$ (see Figure \ref{Kuotrace2}(b); the hexagon $H_{r-1,c-1,n}$ is indicated the the one with bold contour).  Thus, we get
\begin{equation}\label{eq1a}
\M(G-\{u,v,w,s\})=q^{-(2n+r-2)(r-1)/2}\M(H_{r-1,c-1,n}).
\end{equation}

\begin{figure}\centering
\includegraphics[width=12cm]{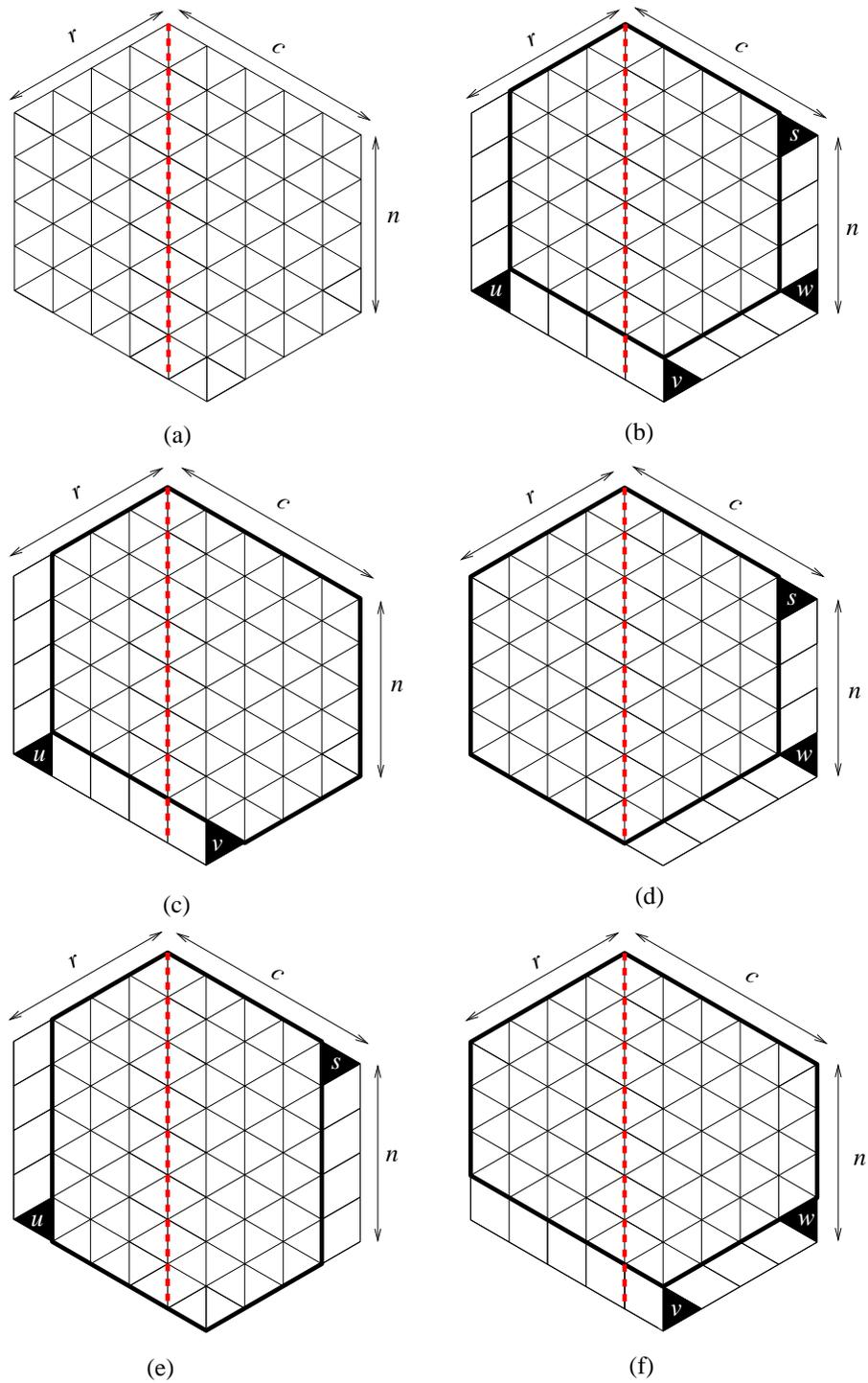}
\caption{Obtaining the recurrence.}
\label{Kuotrace2}
\end{figure}

Similar, we remove the forced lozenges (which have all weight 1) in the region corresponding to $G-\{u,v\}$ and get a the hexagon $H_{r-1,c,n}$ weighted by $\wt$ (see Figure \ref{Kuotrace2}(c)).
\begin{equation}\label{eq1b}
\M(G-\{u,v\})=\M(H_{r-1,c,n}).
\end{equation}

Next, if we remove forced lozenges with weight product $q^{-(2n+r-1)r}$ from the region corresponding to the graph $G-\{w,s\}$, we get the hexagon $H_{r,c-1,n}$ weighted by $\wt$ (see Figure \ref{Kuotrace2}(d)).
\begin{equation}\label{eq1c}
\M(G-\{w,s\})=q^{-(2n+r-1)r}\M(H_{r,c-1,n}).
\end{equation}

The removal of forced lozenges (with weight 1) in the region corresponding to $G-\{u,s\}$ gives us the hexagon $H_{r-1,c-1,n+1}$ (illustrated in  Figure \ref{Kuotrace2}(e)). Next, we divide the weight of each horizontal lozenges along the vertical axis by $\frac{(1-aq^{n+1})}{(1-q^{n+1})}$ to get the weight assignment $\wt$. Since each lozenge tiling of $H_{r-1,c-1,n+1}$ has exactly $d-1$ horizontal lozenges along the vertical axis, this changes the weight of each tiling of by a factor $\frac{(1-aq^{n+1})^{d-1}}{(1-q^{n+1})^{d-1}}$. Then we gat
\begin{equation}\label{eq1d}
\M(G-\{u,s\})=\frac{(1-aq^{n+1})^{d-1}}{(1-q^{n+1})^{d-1}}\M(H_{r-1,c-1,n+1}).
\end{equation}

Finally, we get a weighted hexagon $H_{r,c,n-1}$ from the removal of forced lozenges from the region corresponding to $G-\{v,w\}$ (shown in  Figure \ref{Kuotrace2}(f)). The weight product of the forced lozenges is $q^{-(2n+r-2)(r-1)/2}$. Next, we divide the weight of each horizontal lozenge on the vertical axis by $\frac{(1-q^n)}{(1-aq^{n})}$ and  get back the weight assignment $\wt$. Thus, we obtain
\begin{equation}\label{eq1e}
\M(G-\{v,w\})=q^{-(2n+r-2)(r-1)/2}\frac{(1-q^n)^d}{(1-aq^{n})^d}\M(H_{r,c,n-1}).
\end{equation}
Substituting the equations (\ref{eq1a})--(\ref{eq1e}) into the equation (\ref{Kuoeq}) of  Kuo's Theorem \ref{Kuothm}, we have a recurrence on the tiling generating functions of the hexagons:
\begin{align}\label{recurrence}
\M(H_{r,c,n})\M(H_{r-1,c-1,n})=&q^{-n-r+1}\M(H_{r-1,c,n})\M(H_{r,c-1,n})\notag\\
&+\frac{(1-aq^{n+1})^{d-1}}{(1-q^{n+1})^{d-1}}\frac{(1-q^n)^d}{(1-aq^{n})^d}\M(H_{r-1,c-1,n+1})\M(H_{r,c,n-1}).
\end{align}

\medskip

Our final work is verifying that the expression on the right-hand side of (\ref{main}) satisfies the same recurrence (\ref{recurrence}).
Let $\phi(r,c,n)$ denote this expression, then we need to show that
\begin{align}
\phi(r,c,n)\phi(r-1,c-1,n)=&q^{-n-r+1}\phi(r-1,c,n)\phi(r,c-1,n)\notag\\
&+\frac{(1-aq^{n+1})^{d-1}}{(1-q^{n+1})^{d-1}}\frac{(1-q^n)^d}{(1-aq^{n})^d}\phi(r-1,c-1,n+1)\phi(r,c,n-1)
\end{align}
or
\begin{align}\label{mainr1}
q^{-n-r+1}\frac{\phi(r-1,c,n)}{\phi(r,c,n)}&\frac{\phi(r,c-1,n)}{\phi(r-1,c-1,n)}\notag\\
&+\frac{(1-aq^{n+1})^{d-1}}{(1-q^{n+1})^{d-1}}\frac{(1-q^n)^d}{(1-aq^{n})^d}\frac{\phi(r,c,n-1)}{\phi(r,c,n)}\frac{\phi(r-1,c-1,n+1)}{\phi(r-1,c-1,n)}=1.
\end{align}

We first simplify the fraction $\frac{\phi(r-1,c,n)}{\phi(r,c,n)}$ in the first term on the left-hand side of (\ref{mainr1}) as
\begin{align}
\frac{\phi(r-1,c,n)}{\phi(r,c,n)}&=\frac{q^{-c(r(r-1)/2+(n-1)(r-1))}}{q^{-c(r(r+1)/2+(n-1)r)}}\prod_{j=1}^{c}\prod_{k=1}^{n}\frac{1-aq^{r+j+k-2}}{1-aq^{r+j+k-1}}\notag\\
&=q^{c(n+r-1)}\prod_{j=1}^{c}\prod_{k=1}^{n}\frac{1-aq^{r+j+k-2}}{1-aq^{r+j+k-1}}.
\end{align}
Working similarly on the faction $\frac{\phi(r,c-1,n)}{\phi(r-1,c-1,n)}$ of the first term and multiplying the two fraction up, we obtain
\begin{align}\label{mainr2}
\frac{\phi(r-1,c,n)}{\phi(r,c,n)}\frac{\phi(r,c-1,n)}{\phi(r-1,c-1,n)}&=q^{n+r-1}\prod_{k=1}^{n}\frac{1-aq^{r+c+k-2}}{1-aq^{r+c+k-1}}\notag\\
&=q^{n+r-1}\frac{1-aq^{r+c-1}}{1-aq^{r+c+n-1}}.
\end{align}

We now simplify the  fraction $\frac{\phi(r,c,n-1)}{\phi(r,c,n)}$ in the second term on the left-hand side of (\ref{mainr1}) as
\begin{align}
\frac{\phi(r,c,n-1)}{\phi(r,c,n)}&=a^d q^{cr} \prod_{i=1}^{d} \frac{1-aq^{n}}{1-aq^{n+i-1}}\frac{1-q^{n+i-1}}{1-q^n}\notag\\
&\times \prod_{i=1}^{r}\prod_{j=1}^{c}\frac{1-aq^{i+j+n-2}}{1-aq^{i+j+n-1}}.
\end{align}
Working similarly on the faction $\frac{\phi(r-1,c-1,n+1)}{\phi(r-1,c-1,n)}$ of the second term and multiplying these two fraction up, we get
\begin{align}\label{mainr3}
\frac{\phi(r,c,n-1)}{\phi(r,c,n)}\frac{\phi(r-1,c-1,n+1)}{\phi(r-1,c-1,n)}&=a q^{r+c-1}  \frac{(1-aq^{n})^{d-1}}{(1-q^{n})^{d-1}}\frac{(1-q^{n+1})^{d-1}}{(1-aq^{n+1})^{d-1}}\notag\\
&\times \prod_{i=1}^{r}\prod_{j=1}^{c}\frac{1-aq^{i+j+n-2}}{1-aq^{i+j+n-1}} \prod_{i=1}^{r-1}\prod_{j=1}^{c-1}\frac{1-aq^{i+j+n}}{1-aq^{i+j+n-1}}\notag\\
&=a q^{r+c-1}  \frac{(1-aq^{n})^{d-1}}{(1-q^{n})^{d-1}}\frac{(1-q^{n+1})^{d-1}}{(1-aq^{n+1})^{d-1}}\frac{(1-aq^n)}{(1-aq^{r+c+n-1})}.
\end{align}
By (\ref{mainr2}) and  (\ref{mainr3}), we now have
\begin{align}\label{mainr4}
q^{-n-r+1}\frac{\phi(r-1,c,n)}{\phi(r,c,n)}&\frac{\phi(r,c-1,n)}{\phi(r-1,c-1,n)}\notag\\
&+\frac{(1-aq^{n+1})^{d-1}}{(1-q^{n+1})^{d-1}}\frac{(1-q^n)^d}{(1-aq^{n})^d}\frac{\phi(r,c,n-1)}{\phi(r,c,n)}
\frac{\phi(r-1,c-1,n+1)}{\phi(r-1,c-1,n)}\notag\\
&=\frac{1-aq^{r+c-1}}{1-aq^{r+c+n-1}}+aq^{r+c-1}\frac{1-q^n}{1-aq^{r+c+n-1}}\notag\\
&=1.
\end{align}
This verifies (\ref{mainr1}) and finishes our proof.
\section{Concluding Remarks}
 Gansner \cite{Gansner1, Gansner2} generalized Stanley's trace formula by introducing the \emph{$\ell$-traces} of the plane partition as $tr_{\ell}(\pi)=\sum_{i-j=\ell}\pi_{i,j}$. He showed that
 \begin{equation}
 \sum_{\pi\in\mathcal{P}(r,c)}\prod_{-r<\ell<c}=\prod_{i=0}^{r-1}\prod_{j=0}^{c-1}\left(1-\prod_{\ell=-i}^{j}q_{\ell}\right)^{-1}.
 \end{equation}
Kamioka also proved a finite analogue of Ganser's formula (Theorem 17 in \cite{Kam}). I believe that our method can be used to prove this theorem of Kamioka.

\end{document}